\documentclass[12pt, amsfonts]{amsart}

\usepackage{amsmath,amssymb,amsthm,mathtools,extarrows}
\usepackage{eucal}
\usepackage[top=3.5cm, bottom=3.5cm, left=2.5cm, right=2.5cm, includefoot]{geometry}
\usepackage[dvipdfmx, colorlinks=true, backref=page]{hyperref}
\usepackage[all]{xy}

\numberwithin{equation}{section}

\SelectTips{eu}{12}

\newtheorem{theorem}{Theorem}[section]
\newtheorem{proposition}[theorem]{Proposition}
\newtheorem{lemma}[theorem]{Lemma}
\newtheorem{corollary}[theorem]{Corollary}

\theoremstyle{definition}

\newtheorem{example}[theorem]{Example}

\theoremstyle{remark}

\newcommand{\Z}{\mathbb{Z}}

\newcommand{\R}{\mathbb{R}}

\newcommand{\G}{\mathcal{G}}

\newcommand{\map}{\operatorname{map}}
\newcommand{\Ker}{\operatorname{Ker}}

\makeatletter
\newcommand{\xMapsto}[2][]{\ext@arrow 0599{\Mapstofill@}{#1}{#2}}
\def\Mapstofill@{\arrowfill@{\Mapstochar\Relbar}\Relbar\Rightarrow}
\makeatother

\title{Homotopy types of gauge groups over Riemann surfaces}

\author{Masaki Kameko}
\address{Department of Mathematical Sciences, Shibaura Institute of Technology, 307 Minuma-ku Fukasaku, Saitama-City 337-8570, Japan}
\email{kameko@shibaura-it.ac.jp}

\author{Daisuke Kishimoto}
\address{Department of Mathematics, Kyoto University, Kyoto, 606-8502, Japan}
\email{kishi@math.kyoto-u.ac.jp}

\author{Masahiro Takeda}
\address{Department of Mathematics, Kyoto University, Kyoto, 606-8502, Japan}
\email{takeda.masahiro.87u@st.kyoto-u.ac.jp}



\begin{document}

  \maketitle

  \baselineskip.525cm

  %\setcounter{tocdepth}{1}
  %\tableofcontents

  \parskip .05in
  \parindent .0pt

  %%%%% Section 1 %%%%%

  \begin{abstract}
    Let $G$ be a compact connected Lie group with $\pi_1(G)\cong\Z$. We study the homotopy types of gauge groups of principal $G$-bundles over Riemann surfaces. This can be applied to an explicit computation of the homotopy groups of the moduli spaces of stable vector bundles over Riemann surfaces.
  \end{abstract}

  \baselineskip.525cm

  %%%%% Section 1 %%%%%

  \section{Introduction}\label{Introduction}

  Let $G$ be a compact connected Lie group, and let $P$ be a principal $G$-bundle over a finite complex $X$. The \emph{gauge group} of $P$ is defined to be the topological group of $G$-equivariant self-maps of $P$ which fix $X$. There may be infinitely many distinct principal $G$-bundles over $X$. For example, there are infinitely many bundles when $X$ is an orientable 4-manifold. Each bundle has a gauge group, so there may be potentially certainly infinitely many gauge groups. However, Crabb and Sutherland \cite{CS} showed that these gauge groups have only finitely many homotopy types. Then it has been intensely studied the precise number of homotopy types of gauge groups for specific $G$ and $X$. The study began with simply-connected Lie groups \cite{C,HK,HKST,KK,KKT,KTT,Ko1,T4,T5}, and recently, non-simply-connected cases are also studied as in \cite{HKKS,KKKT,KMST,R}.

  In this paper, we study the homotopy types of gauge groups of principal $G$-bundles over a compact connected Riemann surface, where $\pi_1(G)\cong\Z$. This includes an important case, gauge groups of principal $U(n)$-bundles over a Riemann surface, whose topology was first studied by Atiyah and Bott \cite{AB}. To state the results, we need to define an intrinsic structure of $G$. Suppose $\pi_1(G)\cong\Z$. Then there are a compact connected simply-connected Lie group $H$ and a subgroup $C$ of the center of $S^1\times H$ such that
  \begin{equation}
    \label{decomposition G}
    G\cong(S^1\times H)/C
  \end{equation}
  Note that $H$ is uniquely determined by $G$, but $C$ is not. For example, if $G=S^1\times H$, then $C$ can be any finite subgroup of $S^1\times 1\subset S^1\times H$. We define
  \[
    s(G)=|p_2(C)|
  \]
  where $p_2\colon S^1\times H\to H$. By Theorem \ref{Samelson product} below, we can see that $s(G)$ is independent from the choice of $C$.

  \begin{example}
    Since $U(n)$ is the quotient of $S^1\times SU(n)$ by the diagonal central subgroup isomorphic to $\Z/n$, we have $s(U(n))=n$.
  \end{example}

  Let $X$ be a compact connected Riemann surface. Then there is a one-to-one correspondence between principal $G$-bundles over $X$ and $\pi_2(BG)\cong\Z$. Let $\G_k(X,G)$ denote the gauge group of a principal $G$-bundle over $X$ corresponding to $k\in\Z$. Now we state our results.

  \begin{theorem}
    \label{main 1}
    Let $G$ be a compact connected Lie group with $\pi_1(G)\cong\Z$, and let $X$ be a compact connected Riemann surface. If $(k,s(G))=(l,s(G))$, then $\G_k(X,G)$ and $\G_l(X,G)$ are homotopy equivalent after localizing at any prime or zero.
  \end{theorem}

  We remark that the $p$-localization of a disconnected space will mean the disjoint union of the $p$-localization of path-connected components. For a prime $p$, Theriault \cite{T3} gave a $p$-local homotopy decomposition of $\G_k(X,U(p))$, which implies the converse implication of Theorem \ref{main 1} holds for $G=U(p)$. We will prove the converse implication of Theorem \ref{main 1} holds for other Lie groups.

  \begin{theorem}
    \label{main 2}
    Let $G$ be a compact connected Lie group with $\pi_1(G)\cong\Z$, and let $X$ be a compact connected Riemann surface. If $G$ is locally isomorphic to $S^1\times SU(n)^r$ or $S^1\times SU(4n-2)^s\times Sp(2n-1)^t$, then the following statements are equivalent:
    \begin{enumerate}
      \item $(k,s(G))=(l,s(G))$;

      \item $\G_k(X,G)$ and $\G_l(X,G)$ are homotopy equivalent after localizing at any prime or zero.
    \end{enumerate}
  \end{theorem}

  The homotopy type of a gauge group $\G_k(X,G)$ is closely related with a Samelson product in $G$, as we will see in Section \ref{Gauge groups and Samelson products}. In our context, the Samelson product of a generator of $\pi_1(G)\cong\Z$ and the identity map of $G$ is of particular importance. Then we will prove the following theorem, which is of independent interest.

  \begin{theorem}
    \label{Samelson product}
    Let $G$ be a compact connected Lie group with $\pi_1(G)\cong\Z$, and let $\epsilon$ denote a generator of $\pi_1(G)$. Then the Samelson product $\langle\epsilon,1_G\rangle$ in $G$ is of order $s(G)$.
  \end{theorem}

  Now we consider an application. Gauge groups over a Riemann surface are closely related to the moduli spaces of stable vector bundles over a Riemann surface as follows. Let $X$ be a Riemann surface of genus $g$, and let $M(n,k)$ denote the moduli space of stable vector bundles over $X$ of rank $n$ and degree $k$. Daskalopoulos and Uhlenbeck \cite{DU} showed that there is an isomorphism
  \[
    \pi_i(M(n,k))\cong\pi_{i-1}(\G_k(X,U(n)))
  \]
  for $2<i\le 2(g-1)(n-1)-2$ and $(n,k)\ne(2,2)$. There is a polystable Higgs bundle analog due to Bradlow, Garcia-Prada and Gothen \cite{BGG}. Then we can compute the homotopy groups of these moduli spaces through the following homotopy decomposition.

  \begin{theorem}
    \label{main decomposition}
    Let $G$ be a compact connected Lie group with $\pi_1(G)\cong\Z$, and let $X$ be a compact connected Riemann surface of genus $g$. If $s(G)$ divides $k$, then
    \[
      \G_k(X,G)\simeq G\times(\Omega G)^{2g}\times\Omega^2G.
    \]
    Moreover, the above homotopy equivalence also holds after localizing at $p$ whenever $p$ does not divide $s(G)$.
  \end{theorem}

  The paper is structured as follows Section \ref{Gauge groups and Samelson products} recalls a connection of gauge groups to Samelson products, and then proves Theorems \ref{main 1} and \ref{main decomposition} by assuming Theorem \ref{Samelson product} holds. Section \ref{Samelson products in a Lie group} shows some general results on Samelson products in a Lie group, which will be used for a practical computation. Sections \ref{Classical case} and \ref{Exceptional case} compute the Samelson products in $G$ when $H$ is simple. Finally, Section \ref{Proofs} collects all results so far together to prove Theorems \ref{Samelson product} and Theorem \ref{main 2}.

  \subsection*{Acknowledgement}

  The authors were partly supported by JSPS KAKENHI No. 17K05263 (Kameko), No. 17K05248 (Kishimoto) and No. 21J10117 (Takeda).

  %%%%% Section 2 %%%%%

  \section{Gauge groups and Samelson products}\label{Gauge groups and Samelson products}

  This section recalls a connection of gauge groups to Samelson products, and then Theorems \ref{main 1} and \ref{main decomposition} are proved by assuming Theorem \ref{Samelson product} holds. First, we recall a connection of gauge groups to mapping spaces. Let $G$ be a topological group, and let $P$ be a principal $G$-bundle over a base $X$, which is classified by a map $\alpha\colon X\to BG$. Recall that the gauge group of $P$, denoted by $\G(P)$, is a topological group of $G$-equivariant self-maps of $P$ which fix $X$. Gottlieb \cite{G} proved that there is a natural homotopy equivalence
  \[
    B\G(P)\simeq\map(X,BG;\alpha)
  \]
  where $\map(A,B;f)$ denotes the path component of the space of maps $\map(A,B)$ containing a map $f\colon A\to B$. Then evaluating at the basepoint of $X$ yields a homotopy fibration
  \begin{equation}
    \label{evaluation fibration}
    \map_*(X,BG;\alpha)\to B\G(P)\to BG
  \end{equation}
  where $\map_*(X,BG;\alpha)$ is a subspace of $\map(X,BG;\alpha)$ consisting of basepoint preserving maps. So the gauge group $\G(P)$ is homotopy equivalent to the homotopy fiber of the connecting map
  \[
    \partial_\alpha\colon G\to\map_*(X,BG;\alpha)
  \]
  of the above homotopy fibration.

  Next, we assume $X=S^n$ for $n\ge 1$ and describe the connecting map $\partial_\alpha$. Clearly, there is a homotopy equivalence $\map_*(S^n,BG;\alpha)\simeq\Omega^{n-1}_0G$, where $\Omega^{n-1}_0G$ denotes the path component of $\Omega^{n-1}G$ containing the constant map. Then by adjointing, the connecting map $\partial_\alpha$ corresponds to a map
  \[
    d_\alpha\colon S^{n-1}\wedge G\to G
  \]
  The original definition of Whitehead products in \cite{W} and adjointness of Whitehead products and Samelson products prove the following.

  \begin{lemma}
    \label{connecting map}
    The map $d_\alpha$ is the Samelson product $\langle\bar{\alpha},1_G\rangle$ in $G$, where $\bar{\alpha}\colon S^{n-1}\to G$ is the adjoint of $\alpha\colon S^n\to BG$.
  \end{lemma}

  The following lemma due to Theriault \cite{T2} shows how to identify the homotopy type of a gauge group $\G(P)$ from the order of a Samelson product $\langle\bar{\alpha},1_G\rangle$.

  \begin{lemma}
    \label{homotopy fiber}
    Suppose that a map $f\colon X\to Y$ into an H-space $Y$ is of order $n<\infty$. Then $(n,k)=(n,l)$ implies $F_k{}_{(p)}\simeq F_l{}_{(p)}$ for any prime $p$, where $F_k$ denotes the homotopy fiber of a map $k\circ f\colon X\to Y$
  \end{lemma}

  Finally, we recall a homotopy decomposition of a gauge group. Theriault \cite{T1} showed a homotopy decomposition of a gauge group over principal $U(n)$-bundle over a Riemann surface. We can easily see that his proof works in verbatim for any compact connected Lie group $G$ with $\pi_1(G)\cong\Z$. Then we get:

  \begin{proposition}
    \label{homotopy decomposition}
    Let $G$ be a compact connected Lie group with $\pi_1(G)\cong\Z$, and let $X$ be a compact connected Riemann surface of genus $g$. Then there is a homotopy equivalence
    \[
      \G_k(X,G)\simeq(\Omega G)^{2g}\times\G_k(S^2,G).
    \]
  \end{proposition}

  Now we prove Theorems \ref{main 1} and \ref{main decomposition} by assuming Theorem \ref{Samelson product} holds.

  \begin{proof}
    [Proof of Theorem \ref{main 1}]
    Combine Lemmas \ref{connecting map}, \ref{homotopy fiber}, Proposition \ref{homotopy decomposition} and Theorem \ref{Samelson product}.
  \end{proof}

  \begin{proof}
    [Proof of Theorem \ref{main decomposition}]
    By Lemma \ref{connecting map} and Theorem \ref{Samelson product}, if $k$ is divisible by $s(G)$, then $\G_k(S^2,G)$ is homotopy equivalent to the homotopy fiber of the constant map $G\to\Omega_0G$. So since $\pi_2(G)=0$, $\G_k(S^2,G)\simeq G\times\Omega^2G$. Thus by Proposition \ref{homotopy decomposition}, the proof is done.
  \end{proof}

  %%%%% Section 3 %%%%%

  \section{Samelson products in Lie groups}\label{Samelson products in a Lie group}

  This section shows some criteria for computing Samelson products in a Lie group. For the rest of the paper, we will use the following notation.

  \begin{itemize}
    \item Let $G$ be a compact connected Lie group with $\pi_1(G)\cong\Z$.

    \item Let $\epsilon_G$ denote a generator of $\pi_1(G)\cong\Z$

    \item Let $H$ and $C$ be as in the decomposition \eqref{decomposition G}.

    \item Let $j_H\colon\Sigma H\to BH$ denote the natural map.

    \item Let $p_G\colon S^1\times H\to G$ denote the projection.

    \item Let $p_1\colon S^1\times H\to S^1$ and $p_2\colon S^1\times H\to H$ denote projections.

    \item Let $K=H/p_2(C)$.

    \item Let $q_G\colon G\to K$ and $\bar{q}_K\colon H\to K$ denote projections.
  \end{itemize}

  We will abbreviate $\epsilon_G,j_H,p_G,q_G,\bar{q}_K$ to $\epsilon,j,p,q,\bar{q}$, respectively, if $G,H,K$ are clear from the context. First, we show two properties of the group $C$.

  \begin{lemma}
    \label{p_2(C)}
    The abelian group $p_2(C)$ is cyclic.
  \end{lemma}

  \begin{proof}
    There is a fibration
    \begin{equation}
      \label{fibration}
      S^1\to G\xrightarrow{q}K
    \end{equation}
    and so by the homotopy exact sequence, we can see that$\pi_1(K)\cong p_2(C)$ is a quotient of $\pi_1(G)\cong\Z$. Then $p_2(C)$ is a cyclic group, as stated.
  \end{proof}

  \begin{lemma}
    \label{C}
    We may choose a group $C$ such that $|p_1(C)|=s(G)$.
  \end{lemma}

  \begin{proof}
    Note that $p_2(C)$ is a cyclic group by Lemma \ref{fibration}. We prove an inequality $|p_1(C)|\ge s(G)$ always holds. If $|p_1(C)|<s(G)$, then $C_1=|p_1(C)|C$ is a non-trivial subgroup of the center of $1\times H\subset S^1\times H$. In particular, there is a covering
    \[
      C/C_1\to(S^1\times H)/C_1\to G.
    \]
    Then $\pi_1(G)\cong\Z$ includes a non-trivial finite abelian group $C_1$, which is a contradiction. Thus $|p_1(C)|\ge s(G)$.

    Suppose that $|p_1(C)|>s(G)$. Then $C_2=s(G)C$ is a finite subgroup of $S^1\times 1\subset S^1\times H$. Then $(S^1\times H)/C_2\cong S^1\times H$, implying
    \[
      G\cong(S^1\times H)/C\cong((S^1\times H)/C_2)/(C/C_2)\cong(S^1\times H)/(C/C_2).
    \]
    By definition, $|p_1(C/C_2)|=|p_2(C/C_1)|=|p_2(C)|=s(G)$, and thus the proof is finished.
  \end{proof}

  By Lemma \ref{p_2(C)}, $\pi_1(K)\cong p_2(C)$ is a cyclic group of order $s(G)$. For the rest of this section, we will also use the following notation.

  \begin{itemize}
    \item Let $\bar{\epsilon}_K$ denote a generator of $\pi_1(K)$.
  \end{itemize}

  We will abbreviate it by $\bar{\epsilon}$ if $K$ is clear from the context.

  Next, we show an upper bound and a lower bound for the order of $\langle\epsilon,1_G\rangle$.

  \begin{lemma}
    \label{upper bound}
    The order of $\langle\epsilon,1_G\rangle$, hence $\langle\epsilon,p\rangle$, divides $s(G)$.
  \end{lemma}

  \begin{proof}
    The proof of Lemma \ref{p_2(C)} implies $q\circ\epsilon=\bar{\epsilon}$. Then since $q$ is a homomorphism, we get
    \[
      q_*(s(G)\langle\epsilon,1_G\rangle)=s(G)\langle  q\circ\epsilon,q\rangle=\langle s(G)\bar{\epsilon},q\rangle=0.
    \]
    So since there is a fibration \eqref{fibration}, $s(G)\langle\epsilon,1_G\rangle$ lifts to a map $S^1\wedge G\to S^1$. Since $S^1\wedge G$ is simply-connected, this lift is trivial, and thus $s(G)\langle\epsilon,1_G\rangle$ itself is trivial, completing the proof.
  \end{proof}

  \begin{lemma}
    \label{lower bound}
    The order of $\langle\bar{\epsilon},\bar{q}\rangle$ divides the order of $\langle\epsilon,p\rangle$.
  \end{lemma}

  \begin{proof}
    Let $i\colon H\to S^1\times H$ denote the inclusion. By definition, $q\circ p\circ i=\bar{q}$, and the proof of Lemma \ref{C} implies that $q\circ\epsilon=\bar{\epsilon}$. Then
    \[
      (1\wedge i)^*\circ q_*(\langle\epsilon,p\rangle)=q_*(\langle\epsilon,p\circ i\rangle)=\langle q\circ\epsilon,q\circ p\circ i\rangle=\langle\bar{\epsilon},\bar{q}\rangle
    \]
    and so the proof is done.
  \end{proof}

  Finally, we give a cohomological criterion for the Samelson product $\langle\bar{\epsilon},\bar{q}\rangle$ being non-trivial. For an algebra $A$, let $QA$ denote the module of indecomposables.

  \begin{lemma}
    \label{Steenrod operation}
    Suppose there are $x,y,z\in QH^*(BK;\Z/p)$ and a Steenrod operation $\theta$ satisfying the following conditions:
    \begin{enumerate}
      \item $|y|=2$ and $QH^n(BK;\Z/p)=\langle z\rangle$ for $n>2$;

      \item $\theta(x)$ is decomposable and includes the term $y\otimes z$;

      \item $(\bar{q}\circ j)^*(z)$ is non-trivial and not included in any element of $\theta(H^*(\Sigma H;\Z/p))$.
    \end{enumerate}
    Then the Samelson product $\langle\bar{\epsilon},\bar{q}\rangle$ is non-trivial.
  \end{lemma}

  \begin{proof}
    Suppose that $\langle\bar{\epsilon},\bar{q}\rangle$ is trivial. Let $\hat{\epsilon}\colon S^2\to BK$ and $\hat{q}\colon\Sigma H\to BK$ denote the adjoint of $\bar{\epsilon}$ and $\bar{q}$, respectively. Then by adjointness of Samelson products and Whitehead products, the Whitehead product $[\hat{\epsilon},\hat{q}]$ is trivial, so that there is a homotopy commutative diagram
    \[
      \xymatrix{
        S^2\vee\Sigma H\ar[r]^(.6){\hat{\epsilon}\vee\hat{q}}\ar[d]&BK\ar@{=}[d]\\
        S^2\times\Sigma H\ar[r]^(.6)\mu&BK.
      }
    \]
    By the first condition, $\hat{\epsilon}^*(y)=u$, where $u$ is a generator of $H^2(S^2;\Z/p)\cong\Z/p$. Then by the first and the second conditions, $\mu^*(\theta(x))$ includes the term $u\otimes\hat{q}^*(z)$. Since $\hat{q}=\bar{q}\circ j$, the third condition implies $u\otimes\hat{q}^*(z)\ne 0$. On the other hand, by the third condition, $\theta(\mu^*(x))$ cannot include the term $u\otimes\hat{q}^*(z)$. Thus since $\mu^*(\theta(x))=\theta(\mu^*(x))$, we obtain a contradiction. Therefore $\langle\bar{\epsilon},\bar{q}\rangle$ is non-trivial, completing the proof.
  \end{proof}

  Recall that compact simply-connected simple Lie groups with non-trivial center are
  \[
    SU(n),\quad Sp(n),\quad Spin(n)\quad(n\ge 7),\quad E_6,\quad E_7.
  \]
  Then in the following two sections, we will compute the Samelson product $\langle\epsilon,p\rangle$ for $H$ being one of the above Lie groups.

  %%%%% Section 4 %%%%%

  \section{Classical case}\label{Classical case}

  This section determines the order of the Samelson product $\langle\epsilon,p\rangle$ for $H=SU(n),Sp(n),Spin(n)$.

  %%%%% Subsection 4.1 %%%%

  \subsection{The case $H=SU(n)$}

  First we consider the case $H=SU(n)$.

  \begin{proposition}
    \label{SU(n)}
    If $H=SU(n)$, then $\langle\epsilon,p\rangle$ is of order $s(G)$.
  \end{proposition}

  \begin{proof}
    By Lemma \ref{upper bound}, it suffices to show that the order of $\langle\epsilon,p\rangle$ is a non-zero multiple of $s(G)$. The center of $SU(n)$ is isomorphic to $\Z/n$. Then since $U(n)=S^1\times_{\Z/n}SU(n)$, it follows from Lemma \ref{C} that there is a homomorphism $\rho\colon G\to U(n)$ which is a $\frac{n}{s(G)}$ sheeted covering. Let $\alpha_{2i-1}$ denote a generator of $\pi_{2i-1}(U(n))\cong\Z$ for $i=1,2,\ldots,n$. Then
    \[
      \rho_*(\epsilon)=\frac{n}{s(G)}\alpha_1.
    \]
    On the other hand, it is shown in \cite{Bo} that the order of $\langle\alpha_1,\alpha_{2n-1}\rangle$ is a non-zero multiple of $n$. Since $\rho_*\colon\pi_{2n-1}(G)\to\pi_{2n-1}(U(n))$ is an isomorphism, there is $\tilde{\alpha}\in\pi_{2n-1}(G)$ such that $\rho_*(\tilde{\alpha})=\alpha_{2n-1}$. Then since
    \[
      \rho_*(\langle\epsilon,\tilde{\alpha}\rangle)=\langle\rho_*(\epsilon),\rho_*(\tilde{\alpha})\rangle=\langle\frac{n}{s(G)}\alpha_1,\alpha_{2n-1}\rangle=\frac{n}{s(G)}\langle\alpha_1,\alpha_{2n-1}\rangle,
    \]
    the order of $\rho_*(\langle\epsilon,\tilde{\alpha}\rangle)$ is a non-zero multiple of $s(G)$. Thus since the map $\rho_*\colon\pi_{2n}(G)\to\pi_{2n}(U(n))$ is an isomorphism, the order of $\langle\epsilon,\tilde{\alpha}\rangle$ is a non-zero multiple of $s(G)$ too. Since $p_*\colon\pi_{2n-1}(S^1\times SU(n))\to\pi_{2n-1}(G)$ is an isomorphism, there is $\beta\in\pi_{2n-1}(S^1\times SU(n))$ such that $p\circ\beta=\tilde{\alpha}$. Thus since $(1\wedge \beta)^*(\langle\epsilon,p\rangle)=\langle\epsilon,\tilde{\alpha}\rangle$, the order of $\langle\epsilon,p\rangle$ is a non-zero multiple of $s(G)$, completing the proof.
  \end{proof}

  %%%%% Subsection 4.2 %%%%

  \subsection{The case $H=Sp(n)$}

  Next, we consider the case $H=Sp(n)$. Recall that the center of $Sp(n)$ is isomorphic to $\Z/2$, and the quotient of $Sp(n)$ by its center is denoted by $PSp(n)$. We apply Lemma \ref{Steenrod operation} to the case $H=Sp(n)$. To this end, we compute the mod 2 cohomology of $BPSp(2n)$ in low dimensions.

  \begin{lemma}
    \label{diagonal Sp(2)}
    Let $\Delta=\{\pm(1,\ldots,1)\in Sp(2)^n\}$. Then for $*\le 7$
    \[
      H^*(B(Sp(2)^n/\Delta);\Z/2)=\Z/2[x_2,x_3,x_5]\otimes\bigotimes_{k=1}^n\Z/2[x_{4,k}],\quad Sq^2x_{4,k}=x_2x_{4,k}
    \]
    where $|x_i|=i$ and $|x_{4,k}|=4$.
  \end{lemma}

  \begin{proof}
    Consider the Serre spectral sequence for a homotopy fibration $\R P^\infty\to BSp(2)^n\to B(Sp(2)^n/\Delta)$. Since $H^*(\R P^\infty;\Z/2)=\Z/2[w]$ with $|w|=1$,
    \[
      H^*(\R P^\infty;\Z/2)=\Delta(w,Sq^1w,Sq^2Sq^1w).
    \]
    for $*\le 7$, where $\Delta(a_1,\ldots,a_k)$ denotes the simple system of generators in $a_1,\ldots,a_k$. Clearly, $\tau(w)=x_2$ for a generator $x_2$ of $H^2(B(Sp(2)^n/\Delta);\Z/2)\cong\Z/2$, where $\tau$ denotes the transgression. Then we get $H^*(B(Sp(2)^n/\Delta);\Z/2)$ for $*\le 7$ as stated. It remains to show $Sq^2x_{4,k}=x_2x_{4,k}$. Recall that
    \begin{equation}
      \label{BSO(n)}
      H^*(BSO(n);\Z/2)=\Z/2[w_2,w_3,\ldots,w_n],\quad Sq^iw_j=\sum_{k=0}^i\binom{j+k-i-1}{k}w_{i-k}w_{j+k}
    \end{equation}
    where $w_i$ is the $i$-th Stiefel-Whitney class. Then since $PSp(2)\cong SO(5)$,
    \[
      H^*(BPSp(2);\Z/2)=\Z/2[y_2,y_3,y_4,y_5],\quad Sq^2y_4=y_2y_4
    \]
    where $|y_i|=i$. Let $q_k\colon Sp(2)^n\to Sp(2)$ denote the $k$-th projection for $k=1,2,\ldots,n$. Then $q_k^*(y_2)=x_2$ and $q_k^*(y_4)=x_{4,k}$. Thus we obtain $Sq^2x_{4,k}=x_2x_{4,k}$, completing the proof.
  \end{proof}

  \begin{proposition}
    \label{PSp(2n)}
    For $*\le 7$,
    \[
      H^*(BPSp(n);\Z/2)=\Z/2[x_2,x_3,x_4,x_5],\quad Sq^2x_4=x_4x_2,\quad|x_i|=i.
    \]
  \end{proposition}

  \begin{proof}
    We can compute the mod 2 cohomology of $BPSp(2n)$ in the same way as in the proof of Lemma \ref{diagonal Sp(2)} by considering a homotopy fibration $\R P^\infty\to BSp(2n)\to BPSp(2n)$. Then it remains to show $Sq^2x_4=x_4x_2$. Let $\Delta$ be as in Lemma \ref{diagonal Sp(2)}. Then there is an inclusion $i\colon Sp(2)^n/\Delta\to PSp(2n)$. Clearly, $i^*(x_2)=x_2$ and $i^*(x_4)=x_{4,1}+\cdots+x_{4,n}$. Then we obtain $Sq^2x_4=x_4x_2$ by Lemma \ref{diagonal Sp(2)}.
  \end{proof}

  Now we prove:

  \begin{proposition}
    \label{Sp(n)}
    If $H=Sp(n)$, then $\langle\epsilon,p\rangle$ is of order $s(G)$.
  \end{proposition}

  \begin{proof}
    Since the center of $Sp(n)$ is isomorphic to $\Z/2$, we only consider $G=S^1\times_{\Z/2}Sp(n)$. In this case, $s(G)=2$, so by Lemmas \ref{upper bound}, it suffices to show $\langle\epsilon,p\rangle$ is non-trivial. First, we consider the case $G=S^1\times_{\Z/2}Sp(2n-1)$. The natural inclusion $Sp(2n-1)\to SU(4n-2)$ sends the center of $Sp(2n-1)$ injectively into the center of $SU(4n-2)$. Then we get a homomorphism $G\to S^1\times_{\Z/2}SU(4n-2)$ which is an isomorphism in $\pi_1$. It is well known that the induced map $\pi_{8n-5}(Sp(2n-1))\to\pi_{8n-5}(SU(4n-2))$ is an isomorphism, hence so is $\pi_{8n-5}(G)\to\pi_{8n-5}(S^1\times_{\Z/2}SU(4n-2))$. Then the proof of Proposition \ref{SU(n)} implies that the Samelson product $\langle\epsilon,p\rangle$ is non-trivial.

    Next, we consider $G=S^1\times_{\Z/2}Sp(2n)$. We apply Lemma \ref{Steenrod operation} to $K=PSp(2n)$ by setting $x=z=x_4,\,y=x_2$ and $\theta=Sq^2$. By Proposition \ref{PSp(2n)}, the first and the second conditions of Lemma \ref{Steenrod operation} are satisfied. The proof of Proposition \ref{PSp(2n)} implies $\bar{q}^*(x_4)$ is non-trivial, where $H^4(BSp(2n);\Z/2)\cong QH^4(BSp(2n);\Z/2)\cong \Z/2$. Since the map
    \[
      j^*\colon QH^4(BSp(2n);\Z/2)\to \Sigma QH^3(Sp(2n);\Z/2)
    \]
    is an isomorphism, we have $(\bar{q}\circ j)^*(x_4)\ne 0$. Moreover, for degree reasons, $(\bar{q}\circ j)^*(x_4)$ is not included in any element of $\theta(H^*(\Sigma Sp(2n);\Z/2))$. Then the third condition of Lemma \ref{Steenrod operation} is also satisfied. Thus $\langle\bar{\epsilon},\bar{q}\rangle$ is non-trivial, and so by Lemma \ref{lower bound}, $\langle\epsilon,p\rangle$ is non-trivial too.
  \end{proof}

  %%%%% Subsection 4.3 %%%%%

  \subsection{The case $H=Spin(n)$}

  Finally, we consider the case $H=Spin(n)$. We show some properties of the mod 2 cohomology of $BSpin(n)$ that we are going to use. Recall that the mod 2 cohomology of $BSO(n)$ is given as in \eqref{BSO(n)}.

  \begin{lemma}
    \label{Spin(n) cohomology}
    \begin{enumerate}
      \item The mod 2 cohomology of $BSpin(n)$ is given by
      \[
        H^*(BSpin(n);\Z/2)=\Z/2[u_2,u_3,\ldots,u_n,z]/(u_2,Sq^{2^k}Sq^{2^{k-1}}\cdots Sq^1u_2\mid k\ge 0)
      \]
      where $\bar{q}_{SO(n)}^*(w_j)=u_j$, $|z|=2^h$ for some $h>0$ and $Sq^iu_j$ is computed by replacing $w_j$ with $u_j$ in the formula \eqref{BSO(n)}.

      \item For $2\le i\le n$ with $i\ne 2^k+1$, $j_{Spin(n)}^*(u_i)\ne 0$.
    \end{enumerate}
  \end{lemma}

  \begin{proof}
    (1) is a result of Quillen \cite{Q}. We prove the statement (2). It is well known that $(j')^*(w_i)\ne 0$ for $i=2,3,\ldots,n$, where $j'\colon\Sigma SO(n)\to BSO(n)$ is the natural map. On the other hand, it is shown in \cite{IKT} that $(\Sigma\bar{q}_{SO(n)})^*\circ(j')^*(w_i)\ne 0$. Then for $2\le i\le n$ with $i\ne 2^k+1$,
    \[
      0\ne(\Sigma\bar{q}_{SO(n)})^*\circ(j')^*(w_i)=j^*\circ\bar{q}_{SO(n)}(w_i)=j^*(u_i).
    \]
    Thus the statement (2) is proved.
  \end{proof}

  The following lemma is easily deduced from the formula \eqref{BSO(n)}.

  \begin{lemma}
    \label{Sq w}
    In $H^*(BSO(n);\Z/2)$, we have:
    \begin{enumerate}
      \item If $n\equiv 0,1\mod 4$, then $Sq^2w_i$ for $i=n-3,n-1$ are decomposable and $Sq^2w_{n-1}$ includes the term $w_2w_{n-1}$;

      \item if $n\equiv 2\mod 8$, then $Sq^5w_i$ for $i=n-4,n-9$ are decomposable  and $Sq^5w_{n-4}$ includes the term $w_2w_{n-1}$;

      \item if $n\equiv 6\mod 8$, then $Sq^3w_i$ for $i=n-2,n-4$ are decomposable  and $Sq^3w_{n-2}$ includes the term $w_2w_{n-1}$;

      \item if $n\equiv 3\mod 4$, then $Sq^2w_i$ for $i=n-2,n$ are decomposable and $Sq^2w_n$ includes the term $w_2w_n$.
    \end{enumerate}
  \end{lemma}

  Let $C_n$ denote the center of $Spin(n)$. Then we have:

  \begin{enumerate}
    \item $C_{2n+1}\cong\Z/2$ and $Spin(2n+1)/C_{2n+1}\cong SO(2n+1)$.

    \item $C_{4n+2}\cong\Z/4$ and $Spin(4n+2)/(\Z/2)\cong SO(4n+2)$.

    \item $C_{4n}\cong\Z/2\times\Z/2$, $Spin(4n)/(\Z/2\times 1)\cong SO(4n)$ and $Spin(4n)/(1\times\Z/2)\cong Ss(4n)$.
  \end{enumerate}

  \begin{proposition}
    \label{SO(n)}
    If $H=Spin(n)$ and $K=SO(n)$, then $\langle\epsilon,p\rangle$ is of order $s(G)$.
  \end{proposition}

  \begin{proof}
    We only give a proof for $n$ odd because the case $n$ even is quite similarly proved. We apply Lemma \ref{Steenrod operation} by setting $x=z=w_{n-1},\,y=w_2$ and $\theta=Sq^2$. By Lemma \ref{Sq w}, the first and the second conditions of Lemma \ref{Steenrod operation} are satisfied. By Lemmas \ref{Spin(n) cohomology} and \ref{Sq w}, $(\bar{q}\circ j)^*(w_{n-1})$ is non-trivial and not included in any element of $Sq^2(H^*(\Sigma Spin(n);\Z/2))$. Then the third condition of Lemma \ref{Steenrod operation} is also satisfied, so $\langle\bar{\epsilon},\bar{q}\rangle\ne 0$. Thus since $s(G)=2$, Lemmas \ref{upper bound} and \ref{lower bound} complete the proof.
  \end{proof}

  Let $PO(n)=Spin(n)/C_n$. Then we have:

  \begin{corollary}
    \label{PO(4n+2)}
    If $H=Spin(4n+2)$ and $K=PO(4n+2)$, then $\langle\epsilon,p\rangle$ is of order $s(G)$.
  \end{corollary}

  \begin{proof}
    Let $\bar{\rho}\colon SO(4n+2)\to PO(4n+2)$ denote the projection. Then $\bar{\rho}_*(\bar{\epsilon}_{SO(4n+2)})=2\bar{\epsilon}_{PO(4n+2)}$. Since $S^1\wedge Spin(4n+2)$ is simply-connected, the map
    \[
      \bar{\rho}_*\colon[S^1\wedge Spin(4n+2),SO(4n+2)]\to[S^1\wedge Spin(4n+2),PO(4n+2)]
    \]
    is an isomorphism. By definition, $\bar{q}_{PO(4n+2)}=\bar{\rho}\circ\bar{q}_{SO(4n+2)}$. So by Proposition \ref{SO(n)},
    \[
      2\langle\bar{\epsilon}_{PO(4n+2)},\bar{q}_{PO(4n+2)}\rangle=\bar{\rho}_*(\langle\bar{\epsilon}_{SO(4n+2)},\bar{q}_{SO(4n+2)}\rangle)\ne 0.
    \]
    Then by Lemma \ref{upper bound}, the order of $\langle\bar{\epsilon}_{PO(4n+2)},\bar{q}_{PO(4n+2)}\rangle$ is a non-zero multiple of $s(G)=4$. Thus the proof is complete by Lemmas \ref{upper bound} and \ref{lower bound}.
  \end{proof}

  Let $\Delta$ denote the diagonal subgroup of $\Z/2\times\Z/2$.

  \begin{proposition}
    \label{Ss(4n)}
    If $H=Spin(4n)$ and $p_2(C)=1\times\Z/2,\Delta$, then $\langle\epsilon,p\rangle$ is of order $s(G)$.
  \end{proposition}

  \begin{proof}
    By triality of $Spin(8)$, the case $H=Spin(8)$ is proved by Proposition \ref{SO(n)}. Then we assume $n>2$. The mod 2 cohomology of $PO(4n)$ was determined by Baum and Browder \cite{BB} such that
    \[
    H^*(PO(4n);\Z/2)=\Z/2[v]/(v^{2^r})\otimes\Delta(u_1,\ldots,\widehat{u_{2^r-1}},\ldots,u_{n-1}),\quad\bar{\rho}^*(u_i)=w_i
    \]
    where $4n=2^r(2m+1)$, $|v|=1$ and $|u_i|=i$. The elements $v$ and $u_1$ correspond respectively to generators of subgroups $1\times\Z/2$ and $\Z/2\times 1$ of $C_{4n}\cong\Z/2\times\Z/2$. The Hopf algebra structure of $H^*(PO(4n);\Z/2)$ was also determined such that
    \[
      \bar{\phi}(v)=0\quad\text{and}\quad\bar{\phi}(u_i)=\sum_{j=1}^{i-1}\binom{i}{j}u_j\otimes v^{i-j}
    \]
    where $\bar{\phi}$ is the reduced diagonal map. Let $\gamma\colon PO(4n)^2\to PO(4n)$ denote the commutator map. Since $\bar{\epsilon}(v)\ne 0$, it suffices to show $\gamma^*(x)$ includes the term $v\otimes y$ such that $\rho^*(y)\ne 0$, where $\rho\colon Spin(4n)\to PO(4n)$ denotes the projection. Let $\mu\colon PO(4n)^2\to PO(4n)$ and $\Delta\colon PO(4n)\to PO(4n)^2$ denote the multiplication and the diagonal map, respectively. Let $\iota\colon PO(4n)\to PO(4n)$ be a map given by $\iota(x)=x^{-1}$, and let $T\colon PO(4n)^2\to PO(4n)^2$ be the switching map. Then
    \[
      \gamma=\mu\circ(\mu\times\mu)\circ(1\times 1\times\iota\times\iota)\circ(1\times T\times 1)\circ(\Delta\times\Delta).
    \]
    Let $I_k=\widetilde{H}^*(PO(n)^k;\Z/2)$. Now we compute $\gamma^*(u_i)$:
    \begin{align*}
      u_i&\xmapsto{\mu^*}&&u_i\otimes 1+1\otimes u_i+iu_{i-1}\otimes v\mod I_2^3\\
      &\xmapsto{(\mu\times\mu)^*}&&i(u_{i-1}\otimes v\otimes 1\otimes 1+1\otimes 1\otimes u_{i-1}\otimes v+u_{i-1}\otimes 1\otimes 1\otimes v+1\otimes u_{i-1}\otimes v\otimes 1)\\
      &&&\mod I_1\otimes 1\otimes I_1\otimes 1+1\otimes I_1\otimes 1\otimes I_1+I_4^3\\
      &\xmapsto{(1\times 1\times\iota\times\iota)^*}&&i(u_{i-1}\otimes v\otimes 1\otimes 1+1\otimes 1\otimes u_{i-1}\otimes v-u_{i-1}\otimes 1\otimes 1\otimes v-1\otimes u_{i-1}\otimes v\otimes 1)\\
      &&&\mod I_1\otimes 1\otimes I_1\otimes 1+1\otimes I_1\otimes 1\otimes I_1+I_4^3\\
      &\xmapsto{(1\times T\times 1)^*}&&i(u_{i-1}\otimes 1\otimes v\otimes 1+1\otimes u_{i-1}\otimes 1\otimes v-u_{i-1}\otimes 1\otimes 1\otimes v-1\otimes v\otimes u_{i-1}\otimes 1)\\
      &&&\mod I_1\otimes I_1\otimes 1\otimes 1+1\otimes 1\otimes I_1\otimes I_1+I_4^3\\
      &\xmapsto{(\Delta\times\Delta)^*}&&i(u_{i-1}\otimes y-y\otimes u_{i-1})\mod I_1\otimes 1+1\otimes I_1+I_2^3.
    \end{align*}
    Then for $n$ odd, $\gamma^*(u_7)$ includes the term $v\otimes u_6$, where $\rho^*(u_6)\ne 0$ by Lemma \ref{Spin(n) cohomology}, and for $n$ even, $\gamma^*(u_{11})$ includes the term $v\otimes u_{10}$, where $\rho^*(u_{10})\ne 0$ by Lemma \ref{Spin(n) cohomology}. Thus the Samelson product $\langle\bar{\epsilon},\bar{q}\rangle$ is non-trivial, completing the proof by Lemmas \ref{upper bound} and \ref{lower bound} because $s(G)=2$.
  \end{proof}

  %%%%% Section 5 %%%%%

  \section{Exceptional case}\label{Exceptional case}

  First, we consider the case $H=E_6$.

  \begin{proposition}
    \label{E_6}
    If $H=E_6$, then $\langle\epsilon,p\rangle$ is of order $s(G)$.
  \end{proposition}

  \begin{proof}
    Since the center of $E_6$ is isomorphic to $\Z/3$, we only need to consider the case $G=S^1\times_{\Z/3}E_6$. The mod 3 cohomology of $Ad(E_6)$, which is the quotient of $E_6$ by its center, was determined by Kono \cite{Ko} such that
    \[
      H^*(Ad(E_6);\Z/3)=\Z/3[x_2,x_8]/(x_2^9,x_8^3)\otimes\Lambda(x_1,x_3,x_7,x_9,x_{11},x_{16})
    \]
    such that
    \[
      \bar{\phi}(x_9)=x_8\otimes x_1+x_2\otimes x_7-x_2^3\otimes x_3+x_2^4\otimes x_1\quad\text{and}\quad\bar{q}^*(x_8)\ne 0
    \]
    where $|x_i|=i$. Then by the same computation as in the proof of Proposition \ref{Ss(4n)}, we can see that $\langle\bar{\epsilon},\bar{q}\rangle$ is non-trivial. Thus by Lemmas \ref{upper bound} and \ref{lower bound}, $\langle\epsilon,1_G\rangle$ is of order $s(G)=3$.
  \end{proof}

  Next, we consider the case $H=E_7$. Because the center of $E_7$ is isomorphic to $\Z/2$, we only need to consider the case $G=S^1\times_{\Z/2}E_7$. The Hopf algebra structure of $H^*(Ad(E_7);\Z/2)$ was determined by Ishitoya, Kono and Toda \cite{IKT}, from which we can see that the same computation as $Ad(E_6)$ does not apply to $Ad(E_7)$. So we apply Lemma \ref{Steenrod operation}. Kono and Mimura \cite{KM} showed that the mod 2 cohomology of $BAd(E_7)$ is generated by elements $x_i$ for $i\in\{2,3,6,7,10,11,18,19,34,35,64,66,67,96,112\}$, where $|x_i|=i$. We determine $Sq^2x_6$.

  Let $e_1,e_2,\ldots,e_n$ be the standard basis of $\R^n$. Elements of the spin group $Spin(n)$ are expressed by using $e_1,e_2,\ldots,e_n$. See \cite[Chapter 3]{A}. Recall from \cite[Proposition 4.2]{A} that there are two representations
  \[
    \Delta_{2n}^+,\,\Delta^-_{2n}\colon Spin(2n)\to SU(2^{n-1})
  \]
  such that $\Delta_n^+$ has weights $\frac{1}{2}(\pm x_1\pm x_2\pm\cdots\pm x_n)$ with even numbers of minus signs and $\Delta_n^-$ has weights $\frac{1}{2}(\pm x_1\pm x_2\pm\cdots\pm x_n)$ with odd numbers of minus signs.

  \begin{proposition}
    \label{Spin(4)}
    There is a natural isomorphism
    \[
      Spin(4)\cong\Ker\Delta_4^+\times\Ker\Delta_4^-.
    \]
  \end{proposition}

  \begin{proof}
    There is a product decomposition $ Spin(4)\cong SU(2)\times SU(2)$ such that $\Delta_4^\pm\colon Spin(4)\to SU(2)$ are identified with projections $ SU(2)\times SU(2)\to SU(2)$. Then the statement is proved.
  \end{proof}

  As in \cite[Theorem 6.1]{A}, there is a homomorphism
  \[
    \theta\colon Spin(16)\to E_8
  \]
  whose kernel is $\{1,e_1e_2\cdots e_{16}\}$. Let $\mu\colon Spin(4)\times Spin(12)\to Spin(16)$ denote the homomorphism covering the inclusion
  \[
    SO(4)\times SO(12)\to SO(16),\quad(A,B)\mapsto\begin{pmatrix}A&O\\O&B\end{pmatrix}.
  \]
  Define $\bar{\mu}=\theta\circ\mu\colon Spin(4)\times Spin(12)\to E_8$. Then
  \[
    \Ker\bar{\mu}=\{(1,1),\,(-1,-1),\,(e_1e_2e_3e_4,e_5e_6\cdots e_{16}),\,(-e_1e_2e_3e_4,-e_5e_6\cdots e_{16})\}.
  \]
  Recall from \cite[Chapter 8]{A} that $E_7$ is defined as the centralizer of $\bar{\mu}(\Ker\Delta_4^+\times 1)$ in $E_8$. Then by Lemma \ref{Spin(4)}, there is a homomorphism
  \[
    \hat{\mu}\colon\Ker\Delta_4^-\times Spin(12)\to E_7.
  \]
  Since $-e_1e_2e_3e_4\in\Ker\Delta_4^+$, $\bar{\mu}(-e_1e_2e_3e_4,1)$ commutes with every element of $E_7$ in $E_8$. Moreover, $\bar{\mu}(-e_1e_2e_3e_4,1)=\bar{\mu}(e_1e_2e_3e_4,-1)=\hat{\mu}(e_1e_2e_3e_4,-1)$, which belongs to $E_7$ and is not the unit of $E_7$. Then we obtain:

  \begin{proposition}
    \label{Z(E_7)}
    The center of $E_7$ is $\{1,\hat{\mu}(e_1e_2e_3e_4,-1)\}$.
  \end{proposition}

  Let $L=(\Ker\Delta_4^-\times Spin(12))/\{(1,1),(e_1e_2e_3e_4,-1)\}$. Then by Proposition \ref{Z(E_7)}, there is a map
  \[
    \rho\colon L\to Ad(E_7)
  \]
  which is an isomorphism in the second mod 2 cohomology.

  \begin{lemma}
    \label{Sq E_7}
    In $H^*(BAd(E_7);\Z/2)$, $Sq^2x_6$ is decomposable and includes the term $x_2x_6$.
  \end{lemma}

  \begin{proof}
    By \cite{KM,KMS}, $(\bar{\mu}\circ(1\times\bar{q}))^*(x_6)$ includes the term $1\otimes u_6$, where $u_i$ is as in Lemma \ref{Spin(n) cohomology}. Note that the composition
    \[
      Spin(12)\to\Ker\Delta_4^-\times Spin(12)\to L\xrightarrow{q_2} SO(12)
    \]
    is the natural projection, where $q_2$ is the second projection. Then by degree reasons,
    \[
      \rho^*(x_6)+a\rho^*(x_2)^3+b\rho^*(x_3)^2=q_2^*(w_6)
    \]
    for some $a,b\in\Z/2$. On the other hand, $q_2^*\colon H^2(BSO(12);\Z/2)\to H^2(BL;\Z/2)$ is an isomorphism, implying $\rho^*(x_2)=q_2^*(w_2)$. Then since $Sq^2w_6=w_2w_6$ by \eqref{BSO(n)} and $Sq^2x_6$ is decomposable by degree reasons, $Sq^2x_6$ is decomposable and includes the term $x_2x_6$, as stated.
  \end{proof}

  We are ready to prove:

  \begin{proposition}
    \label{E_7}
    If $H=E_7$, then $\langle\epsilon,p\rangle$ is of order $s(G)$.
  \end{proposition}

  \begin{proof}
    As mentioned above, we only need to consider $G=S^1\times_{\Z/2}E_7$. We apply Lemma \ref{Steenrod operation} by setting $x=z=x_6,\,y=x_2$ and $\theta=Sq^2$. By Lemma \ref{Sq E_7}, the first and the second conditions of Lemma \ref{Steenrod operation} are satisfied. As in \cite{KMS}, $\bar{q}^*(x_6)$ is a generator of $H^6(BE_7;\Z/2)$ such that $(\bar{q}\circ j)^*(x_6)$ is non-trivial. Then by degree reasons, the third condition of Lemma \ref{Steenrod operation} is also satisfied, implying $\langle\bar{\epsilon},\bar{q}\rangle$ is non-trivial. Since $s(G)=2$, the proof is complete by Lemmas \ref{upper bound} and \ref{lower bound}.
  \end{proof}

  %%%%% Section 5 %%%%%

  \section{Proofs of Theorems \ref{main 2} and \ref{Samelson product}}\label{Proofs}

  This section proves Theorems \ref{main 2} and \ref{Samelson product}. First, we prove Theorem \ref{Samelson product}.

  \begin{proof}
    [Proof of Theorem \ref{Samelson product}]
    Suppose $H\cong H_1\times\cdots\times H_k$, where each $H_i$ is a simple Lie group. Let $r_i\colon S^1\times H\to S^1\times H_i$ be the projection, and let $G_i=(S^1\times H_i)/(r_i(C))$  for $i=1,2,\ldots,k$. By definition, $s(G)$ is the least common multiple of $s(G_1),\ldots,s(G_k)$.

    Let $\bar{r}_i\colon G\to G_i$ and $\iota_i\colon S^1\times H_i\to S^1\times H$ denote the projection and the inclusion, respectively. Then $\bar{r}_i\circ\epsilon_G=\epsilon_{G_i}$ and $\bar{r}_i\circ p_G\circ\iota_i=p_{G_i}$, so that
    \[
      (1\wedge\iota_i)^*\circ(\bar{r}_i)_*(\langle\epsilon_G,p_G\rangle)=\langle\bar{r}_i\circ\epsilon_G,\bar{r}_i\circ p_G\circ\iota_i\rangle=\langle\epsilon_{G_i},p_{G_i}\rangle.
    \]
    Thus the order of $\langle\epsilon_G,p_G\rangle$ is a non-zero multiple of the order of $\langle\epsilon_{G_i},p_{G_i}\rangle$. So by Propositions \ref{SU(n)}, \ref{Sp(n)}, \ref{SO(n)}, \ref{E_6} and \ref{E_7}, the order of $\langle\epsilon_G,p_G\rangle$ is a non-zero multiple of $s(G_i)$ for $i=1,2,\ldots,k$, hence so is $\langle\epsilon_G,1_G\rangle$. Therefore by Lemma \ref{upper bound}, the proof is complete.
  \end{proof}

  Next, we prove Theorem \ref{main 2}.

  \begin{proof}
    [Proof of Theorem \ref{main 2}]
    We only prove the case $H=SU(n)^r$ because the case $H=SU(4n-2)^s\times Sp(2n-1)^t$ is proved analogously. The implication (1) $\Rightarrow$ (2) follows from Theorem \ref{main 1}. We prove the implication (2) $\Rightarrow$ (1). Let $\partial_k\colon G\to\map_*(S^2,BG;k)\simeq\Omega_0G$ be as in Section \ref{Gauge groups and Samelson products}, and let $q_i\colon H\to SU(n)$ be the projection onto the $i$-th $SU(n)$. Then by Lemma \ref{connecting map}, the proof of Proposition \ref{SU(n)} implies that the image of the map
    \[
      (\partial_k)_*\colon\pi_{2n-1}(G)\to\pi_{2n-1}(\Omega_0G)
    \]
    is isomorphic to $\prod_{i=1}^r\Z/\frac{n!}{(k,|q_i(C)|)}$, where $\pi_{2n-1}(\Omega_0G)\cong(\Z/n!)^r$. By \eqref{evaluation fibration}, there is an exact sequence
    \[
      0\to\prod_{i=1}^r\Z/\tfrac{n!}{(k,|q_i(C)|)}\to(\Z/n!)^r\to\pi_{2n-1}(B\G_k(S^2,G))\to\pi_{2n-1}(BG)\cong\pi_{2n-1}(BSU(n)^r)=0.
    \]
    Then since $\pi_{2n-1}(B\G_k(S^2,G))\cong\pi_{2n-2}(\G_k(S^2,G))$, $\pi_{2n-2}(\G_k(S^2,G))\cong\prod_{i=1}^r\Z/(k,|q_i(C)|)$. So if $\G_k(X,G)\simeq\G_l(X,G)$, then $\pi_{2n-2}(\G_k(S^2,G))\simeq\pi_{2n-2}(\G_l(S^2,G))$, implying
    \[
      (k,|q_1(C)|)\cdots(k,|q_r(C)|)=(l,|q_1(C)|)\cdots(l,|q_r(C)|).
    \]
    As in the proof of Theorem \ref{main 2}, $s(G)$ is the least common multiple of $|q_1(C)|,\ldots,|q_r(C)|$. Then it is easy to see that the above equality implies $(k,s(G))=(l,s(G))$, completing the proof.
  \end{proof}

\end{document}